\documentclass[onecolumn]{autart}    
\usepackage{graphicx}          
\graphicspath{{./imgs/}}

\usepackage{amsfonts,amsmath,amssymb}
\usepackage{fancybox}
\usepackage{graphicx}
\graphicspath{{./imgs/}}
\usepackage{ifthen}
\usepackage{bm}
\usepackage{float}
\usepackage{xcolor} 
\usepackage{pifont}
\usepackage{lscape}
\usepackage{tabularx}
\usepackage{multirow}
\usepackage{textcomp} 

\usepackage{hyphenat}
\usepackage{placeins}

\usepackage{mathtools}

\def\dd{\mathrm{d}}

\newtheorem{property}[thm]{Property}

\usepackage{stmaryrd}

\begin{document}

\begin{frontmatter}

  \title{Finite Dimensional Approximation to Muscular Response in Force-Fatigue Dynamics using Functional 
  Electrical Stimulation} 

  \thanks[footnoteinfo]{This paper was not presented at any IFAC meeting.
  It benefited from the support
    of the FMJH Program PGMO and from the support of EDF; Thales, Orange and the 
    authors are partially supported by the Latex AMIES.
Corresponding author J.~Rouot. Email. jeremy.rouot@yncrea.fr.
  }

  \author[Dijon]{Toufik Bakir}\ead{toufik.bakir@u-bourgogne.fr},     
  \author[INRIA]{Bernard Bonnard}\ead{bernard.bonnard@u-bourgogne.fr},              
  \author[Dijon,INRIA]{Sandrine Gayrard}\ead{sandrine.gayrard@grenoble-inp.org},  
  \author[Brest]{Jérémy Rouot}\ead{jeremy.rouot@yncrea.fr}  

  \address[Dijon]{Univ. Bourgogne Franche-Comté, ImViA Laboratory EA 7508, 9 avenue Alain Savary, Dijon, France}  
  \address[INRIA]{INRIA, 2004 Route des Lucioles, 06902 Valbonne, France}
  \address[Brest]{L@bisen, Vision-AD Team, Yncrea Ouest, 20 Rue Cuirassé Bretagne, Brest, France}


  \begin{abstract}                         
    Recent dynamical models, based on the seminal work of V. Hill, allow to predict 
    the muscular response to functional electrostimulation (FES), in the isometric and non-isometric cases.
    The physical controls are modeled as Dirac pulses and lead to a sampled-data
    control system, sampling corresponding to times of the stimulation, 
    where the output is the muscular force response.
    Such a dynamics is suitable to compute optimized controls aiming to produce a constant force or 
    force strengthening, but is complex for real time applications.
    The objective of this article is to construct a finite dimensional approximation of this
    response to provide fast optimizing schemes, in particular for
    the design of a smart electrostimulator for muscular
    reinforcement or rehabilitation.
    It is an on-going industrial project based on force-fatigue models, validated by experiments.
    Moreover it opens the road to application of optimal control to track a reference trajectory in the joint
    angular variable to produce movement in the non-isometric models.
  \end{abstract}

\end{frontmatter}

%\tableofcontents

\section{Introduction}
Based on the seminal work of V. Hill \cite{Gesztelyi2012}, recent mathematical models 
(validated by experiments) allow to predict the force response
to external stimulation.
They are presented and discussed in details in \cite{wilson2011} in the non-fatigue isometric case.
They were extended in particular by Ding et al. \cite{Ding2000,Ding2002,Ding2002b} to take 
into account the muscular fatigue due to a long stimulation period
and later in \cite{Marion2013} to analyze the joint angular variable response, in the 
non-isometric case aiming to produce movements.
Such models contain two basic nonlinearities to model the complexity of the dynamics.
Fist of all, the ionic conduction and the nonlinear effect of successive 
pulses on the Ca$^{2+}$-concentration. 
Second, the nonlinear dynamics relating the
muscular force response to such concentration, 
modeled by the Michaelis-Menten-Hill functions \cite{menten1913}.

For each train of pulses, due to digital constraints, only a finite number of pulses can be applied 
and from the optimal control point of view, the problem fits into the frame of optimal sampled-data control
problems, studied in particular in \cite{Bourdin2016} to derive Pontryagin necessary conditions.
They can be analyzed 
to determine optimized train pulses and compared with direct optimizing schemes. 
A previous series of articles described the optimal control problems related to track a reference force
or force strengthening, the control being either the interpulse $I_i=t_i-t_{i-1}$ between two 
successive pulses or the amplitude of each pulse.
In particular, model predictive control (MPC) method is presented in \cite{Bakir2019} aiming the use 
of online optimized closed loop control in the applications using force-fatigue model, as suggested in \cite{Doll2015}.
Direct methods vs indirect methods based on Pontryagin type necessary conditions are discussed and numerically
implemented in \cite{Bakir2020} for the isometric case or in \cite{bonnard2020} for the non-isometric case.

The conclusion of aforementioned articles is that the nonlinear dynamics is computationally expensive in the numerical
integration procedure and a challenging task is to reduce this time for real time computation in the 
applications.
This article is motivated by the design of a smart electrostimulator, where the Ding et al. model
is used to adjust automatically the frequency and the amplitude of the stimulations and to compute the 
sequence of stimulations and rest periods adapted to the task of the training program, e.g. endurance
program or force strengthening program.
The objective of this article being to bypass the computational difficulty by constructing 
a finite dimensional approximation of the force response,
depending upon the parameters of each individual, which can be online estimated, aiming a 
real time computation of the optimized amplitudes and times, for each training program.
Note that this approximation has been coded and the application scheme to the smart 
electrostimulator is briefly presented in the final section.

The article is organized as follows.
In section \ref{sec2}, the mathematical model called the Ding et al. model \cite{Ding2000,Ding2002,Ding2002b}
is presented and the main properties of the dynamics are described, reflecting the features of the 
muscular activity. Hence our analysis can be applied to different models discussed in \cite{wilson2011}.
The section \ref{sec3} presents the optimization problems, in relation with muscular dynamics and oriented towards
the design of a smart electrostimulator, where each training program must be translated into an
optimization problem.
Section \ref{sec4} is the technical contribution of this article, 
that is the construction of the approximation for real time computation.
In section \ref{sec5}, we present some numerical simulations aiming to validate the approximation
and the optimizing scheme.
In the final section \ref{sec6}, we outline the application to the design of the smart electrostimulator.
It is based on a nonlinear output tracking \cite{hirschorn1987,hirschorn1988,Isidori1989} 
as a general theoretical frame and is applied to produce a constant force in our situation, but it can
be extended to the non-isometric case to obtain any reference force.
The conclusion indicates directions to complete our analysis, related to online parameters estimation
of the problems \cite{Stein2013,wilson2011} and MPC-methods \cite{Richalet1993,Boyd2010} suitable for 
practical applications.

\section{Mathematical model and main properties} \label{sec2}

We present the Ding et al. model force-fatigue model \cite{Ding2000,Ding2002,Ding2002b}, 
extension of the original Hill model \cite{Gesztelyi2012}.

\subsection{Ding et al. force-fatigue model \cite{Ding2000,Ding2002,Ding2002b}}

The FES input $u$ over a pulse train $[0,T]$ is given by
\begin{equation}
  u(t) = \sum_{i=0}^n \eta_i \delta(t-t_i), \quad t\in [0,T],
  \label{eq:dirac-peigne}
\end{equation}
where $0=t_0<t_1<\dots< t_n<T$ are the impulsion times with $n\in \mathbb{N}$ being fixed and
$\eta_i$ being the amplitudes of each pulses, which are convexified by taking $\eta_i\in [0,1]$,
$\delta(t-t_i)$ denoting the Dirac function at time $t_i$.

Such physical control will provide the FES-signal denoted by $E(t)$, which drives the force response
using electrical conduction and its dynamics is given by 
\begin{equation}
  \dot E(t) + \frac{E(t)}{\tau_c} = \frac{1}{\tau_c} \sum_{i=0}^n R_i \eta_i \delta(t-t_i), \ a.e.\ t\in [0,T],
  \label{eq:Edot}
\end{equation}
with $E(0)=0$, depending upon the time response $\tau_c$ and the scaling function
$R_i$ defined by
\[R_i = \left\{
    \begin{array}{ll}
      1  & \text{ if } i=0
      \\
      1 + (\bar R-1)\, e^{-(t_i - t_{i-1})/\tau_c} & \text{ otherwise },
    \end{array}
\right.\]
which codes the memory effect of successive muscle contractions and is associated to {\it tetanus} \cite{wilson2011}.

The first result is: 
\begin{lem}
  Integrating \eqref{eq:Edot}, one gets 
  \begin{equation*}
    E(t) = \frac{1}{\tau_c} \sum_{i=0}^n R_i e^{-\frac{t-t_i}{\tau_c}} \eta_i H(t-t_i),
    \label{eq:E}
  \end{equation*}
  where $H$ is the Heaviside function and the FES signal depends upon two parameters $(\tau_c,\bar R)$.
\end{lem}

\begin{defn}
  Consider a control system of the form: $\frac{\dd x}{\dd t}=f(x,u)$ where $x\in \mathbb{R}^n$,
  $u\in U\subset \mathbb{R}^m$.
  It is said {\it permanent} if $u$ is a measurable bounded mapping valued in $U$.
  It is called a sampled-data control system if the set of controls is restricted to the set of 
  piecewise constant mappings $[u_0,u_1,\dots,u_n]$, $u_i\in U$ over a set of times 
  $t_0=0<t_1<\dots <t_n<T$, where $n$ is a fixed integer.
  \label{def:control}
\end{defn}

Our problem can be formulated in the sampled-data control frame.
One can write from \eqref{eq:E}, 
\begin{equation*}
  E(t) = \frac{e^{-t/\tau_c}}{\tau_c}\ \sum_{i=0}^n R_i\eta_i e^{-t_i/\tau_c}\, H(t-t_i)
  = \sum_{i=0}^n u_i(t),
  \label{eq:cn}
\end{equation*}
where $u_i(t)$ is the effect of the pulse $\eta_i \delta(t-t_i)$ on the linear dynamics \eqref{eq:Edot}.
One introduces the following.

\begin{defn}
  For each $i$ in $\{0,\dots,n\}$, 
  the restriction of $u_i$ to $[t_i,t_{i+1}]$ is called {\it the head} and the 
  restriction to $[t_{i+1},T]$ is called {\it the tail}.
\end{defn}

Clearly the FES-input is in the generalized frame of sampled data control system, provided 
we take into account the time-dependence and the phenomenon of tetanus.
Observe also that each impulse have an effect on the whole train $[0,T]$.

The FES signal drives the evolution of the electrical conduction according to the linear 
dynamics describing the evolution of $Ca^{2+}$-concentration $c_N$:
\begin{equation}
  \dot c_N(t) + \frac{c_N(t)}{\tau_c} = E(t)
  \label{eq:cNdot}
\end{equation}
and integrating the (resonant) system with $c_N(0)=0$ yields the following:

\begin{prop}
  The concentration is 
  \begin{equation}
    c_N(t) = \frac{1}{\tau_c} \sum_{i=0}^n R_i \eta_i (t-t_i)\,e^{-\frac{t-t_i}{\tau_c}}\, H(t-t_i),
    \label{eq:cN}
  \end{equation}
  which are the superposition of lobes of the form 
  \begin{equation}
    \ell_i(t) = \frac{1}{\tau_c} R_i \eta_i (t-t_i)\,e^{-\frac{t-t_i}{\tau_c}},
    \label{eq:ci}
  \end{equation}
  whose restriction to $[t_i,t_{i+1}]$ forms the head of the corresponding lobe.
  \label{prop:cN}
\end{prop}

\begin{pf}
  Apply  a time translation to the initial lobe with $R_0=1$.
\end{pf}

Introducing the functions
\begin{equation}
  m_1(t) = \frac{c_N(t)}{K_m+c_N(t)},\quad m_2(t) = \frac{1}{\tau_1 + \tau_2\, m_1(t)},
  \label{eq:m1m2}
\end{equation}
where $m_1$ is the {\it Michaelis-Menten-Hill function} \cite{menten1913}, the force response satisfies
the Hill dynamics
\begin{equation}
  \dot F(t) = -m_2(t)\, F(t) + m_1(t) A,
  \label{eq:dotF}
\end{equation}
and where $A,K_m,\tau_1,\tau_2$ being additional parameters and we denote by $\Lambda = 
(\bar R, \tau_c, A,K_m,\tau_1,\tau_2)$ the whole set of parameters. \\
The model can be extended to take into account the fatigue.
Using sensitivity analysis from \cite{bonnard2020}, we shall restrict
our study to the case of the {\it force-fatigue Ding et al. model} with the single 
equation:
\begin{equation}
  \dot A(t) = -\frac{A(t) - A_{rest}}{\tau_{fat}} + \alpha_A \, F(t)
  \label{eq:dotA}
\end{equation}
for all $t\in [0,t_f]$, where $t_f$ is the total time and $A(0)=A_{rest}$
corresponds to the fixed value of $A$ for the non fatigue model.
This leads to introduce additional parameters $\tau_{fat}, \alpha_A$.
Typical parameters values used in this article to validate numeric simulations 
are reported in Table \ref{tb:params} .

\subsection{Mathematical rewriting}
For the previous force-fatigue model and for the sake of the analysis, the model is rewritten
as the control system:
\begin{equation*}
  \dot x(t) = g(x(t)) + b(t)\, \sum_{i=0}^n  G(t_{i-1},t_i)\eta_i\,H(t-t_i)\, \bm{e} 
  \label{eq:dotx}
\end{equation*}
with $x=(x_1,\dots,x_8)^\intercal =(c_N,F,A,\bar R, \tau_c, \tau_1,\tau_2,K_m)^\intercal$
which splits into state variables $(c_N,F,A)$ and fatigue parameters $\Lambda =
(\bar R, \tau_c,\tau_1,\tau_2,K_m)$ satisfying the dynamics 
$
  \dot \Lambda (t) = 0,
$
 the system being integrated with the initial condition $x_0=(0,0,A_{rest}, \Lambda(0))^\intercal$
and
\begin{equation*}
  \begin{array}{ll}
    \bm{e}=(1,0\dots,0)^\intercal, & b(t)=\frac{1}{\tau_c}e^{-t/\tau_c},
    \\
    G(t_{i-1},t_i) = (\bar R-1) e^{t_{i-1}/\tau_c} + e^{t_i/\tau_c},
  \end{array}
\end{equation*}
where $t_{-1}=-\infty, t_0=0$ and $t_{n+1}=T$.

This leads to a control system of the form 
\[
  \dot x(t) = g(x(t)) + b(t) \, \sum_{i=0}^n G(t_{i-1},t_i) \eta_i H(t-t_i)\ \bm{e}
\]
with $x(0)=x_0$.

The variable $\sigma=(t_1,\dots,t_n,\eta_0,\eta_1,\dots,\eta_n)$ denotes the finite dimensional input-space
with the constraints
\begin{equation*}
  \begin{aligned}
    &\eta_i \in [0,1],\ i=0,\dots , n 
    \\
    &0<t_1<\dots< t_n<T, \quad t_i-t_{i-1}\ge I_{\min}, \quad i=1,\dots,n\ ,
  \end{aligned}
\end{equation*}
where $I_{\min}$ is the smallest admissible interpulse.

Moreover the control is observed using the following observation function
\begin{equation}
  y(t) = h(x(t)),
  \label{eq:obs-function}
\end{equation}
and $h:x \mapsto (F,A)$ serves as a direct measure of the muscular force and the fatigue variable.

The following properties are straightforward but crucial in our analysis.
\begin{prop}
  The input-output mapping $\sigma \mapsto y(t)$ is piecewise smooth over $[0,T]$ and smooth
  if $t\neq t_i,\ i=0,\dots, n$ (impulse times).
  \label{prop:ysmooth}
\end{prop}

\begin{prop}
  For the non fatigue model, the force response can be integrated up to a time reparameterization as 
  \begin{equation}
    F(s) = \int_0^s e^{u-s} m_3(u)\, \dd u
    \label{eq:Fs}
  \end{equation}
  with 
  \begin{equation}
    m_3(s) = A \frac{m_1(s)}{m_2(s)}, \quad \dd s = m_2(t)\,\dd t.
    \label{eq:m3s}
  \end{equation}
  \label{prop:Fs}
\end{prop}
\begin{pf}
  Hill dynamics \eqref{eq:dotF} is rewritten as 
  \[
    \frac{\dd F}{\dd s} = m_3(s) - F(s)
  \]
  and this linear dynamics can be integrated using Lagrange formula with $F(0)=0$.
  This proves the assertion. $\hfill \qed$
\end{pf}

\section{Optimization problems related to the design of the electrostimulator} \label{sec3}

\subsection{Standard electrostimulators vs smart electrostimulators}

The standard commercial electrostimulators apply a sequence of pulses trains and rest periods,
where on each train $[0,T]$ the user only imposes the amplitude of the pulses trains and the constant frequency
is related to training program, typically low frequency for endurance program and high frequency
for force strengthening program.
Our aim is to introduce optimization problems related to the design of a smart electrostimulator,
which will be discussed in Section \ref{sec6}.

\subsection{Optimization problems}
\label{secOCP}

\subsubsection{The punch program}

In this case, our aim is to optimize the force at the end of the train over each train $[0,T]$.
This leads to:

{\small {\bf OCP1:} $\max_{\sigma} F(T)$.}\\
In this case, the amplitudes can be held at the constant maximal values  $\eta_i=1,\ i=0,\dots,n$ 
and the optimization variables are the impulse times:
\[
  0=t_0<t_1<\dots<t_n<T.
\]
Since one considers a single train, the force model is sufficient.

\subsubsection{The train endurance program}

We consider a single train $[0,T]$ on which the model is the force model and the 
corresponding problem is 

{\small {\bf OCP2:} $\min_{\sigma} \displaystyle \int_0^T |F(t)-F_{ref}|^2 \, \dd t$.}\\
Here, the amplitudes are appended to the impulse times to form the optimization variables
and we use the convexified amplitudes constraints: $\eta_i \in [0,1],\ i=0,\dots,n$.

The force reference has to be adjusted in relation with the user and can be set to
$F_{\max}/k$, where $k$ is a suitable positive number greater than 1 and $F_{\max}$
is deduced from {\bf OCP1}.

\subsubsection{The  endurance program}

We consider an interval $[0,t_f]$, where $t_f$ is the total training period formed
by sequences of stimulation and rest periods. 
In this case, one must use a force-fatigue model and we take into account the constraint
$A\in [A_{rest},A_{rest}/k'']$, where $S=A_{rest}/k''$ corresponds to a {\it fatigue threshold} since,
as reported in \cite{Ding2000}, if the user is exhausted, the force signal is
{\it totally noisy}.
Moreover in the case of exhaustion, a large rest period is required.
This constraint can be penalized as follows

{\small {\bf OCP3:} $\min_{\sigma} \displaystyle \int_0^{t_f} |F(t)-F_{ref}|^2 \, \dd t \, +\,
w_1\, \int_0^{t_f} |A(t)-A_S|^2 \, \dd t$},\\ 
where $A_S$ is related to $S$, while $w_1$ is a weight parameter.

\section{Construction of an integrable model for real time application} \label{sec4}

\subsection{Mathematical analysis of $c_N$}

  A pulses train is defined by a finite sequence of impulse times  $\sigma = (t_i)_{0\le i\le n}$
  such that $t_0<\dots<t_n$ and we extend it to the left by $t_{-1}=-\infty$ and to the right by
  $t_{n+1}=T$.
  The response $c_N$ can be decomposed as a sum of lobes defined as follows.

\begin{defn}
  A lobe at $t_k$ is the representative curve of the function $\ell_k:\mathbb{R}\ni t\mapsto R_k\eta_k\, \frac{t-t_k}{\tau_c}\, e^{-(t-t_k)/\tau_c}\, H(t-t_k)$. 
  \label{def:lobe}
\end{defn}

\begin{property}
  \begin{itemize}
    \item   A lobe at $t_k$ reaches its maximum 
      at $t=t_k+\tau_c$ and is equal to $R_k\eta_k/e$.  It is strictly increasing on $[t_k, t_k+\tau_c]$ and
      strictly decreasing $[t_k+\tau_c, t_{k+1}]$.
    \item  $\ddot \ell_k$ has a unique zero at $t_k + 2\tau_c$ and therefore, 
    $\ell_k$ is concave on $[t_k,t_k+2\tau_c]$ and convex on $[t_k+2\tau_c,t_{k+1}]$.
    \item  $\ell$ defines a density probability function
      and more than $95\%$ of the values lie in $[t_k,t_k+5\tau_c]$.
      For all $t\ge t_{k}+5\tau_c$, $|\ell_k(t)|\le R_k \eta_k\, 5e^{-5}$.
      \label{prop:ell}
  \end{itemize}
\end{property}

\begin{prop}
  Denote for $k=0,\dots,n$, $c_N^k = c_{N_{\mid [t_k,t_{k+1}]}}$ and 
  $\displaystyle \bar c_N^k = \frac{1}{t_{k+1}-t_k}\ \int_{t_k}^{t_{k+1}} c_N(t)\, \dd t$.
  We have
  \begin{equation}
    \begin{aligned}
      & c_N^k= \sum_{i=0}^k R_i \eta_i \frac{t-t_i}{\tau_c} e^{-(t-t_i)/\tau_c},
      \\
      & \bar c_N^k= \frac{1}{t_{k+1}-t_k} \, \sum_{i=0}^k R_i \eta_i
      \left( \chi_i(t_k)-\chi_i(t_{k+1})          \right),
    \end{aligned}
    \label{eq:cnk}
  \end{equation}
  where $\chi_i(t)=
  e^{-(t-t_i)/\tau_c}
  \left(\tau_c + t-t_i\right)$.
  \label{prop:cnk}
\end{prop}

\begin{defn}
  The polynomial-exponential category for (piecewise) smooth functions 
  $[0,T] \mapsto \mathbb{R}$ is the category generated by sums, products of polynomials
  $P(t)$ and exponential mappings to generate exponential-polynomials:
  $\sum_n P_n(t) e^{\lambda_n t}$.
  This category is stable with respect to derivation and integration.
  \label{def:category}
\end{defn}

Using proposition \ref{prop:cN} one has: 
\begin{lem}
  For $t\neq t_i$, $c_N(t)$ is in the polynomial-exponential category. 
  Moreover, the coefficients are linear with respect to $\eta_i$ and polynomial-exponential with respect to $t_i$.
  \label{lem:category}
\end{lem}

We introduce the notion of $p$-persistent pulses related to the case where $p$th successive lobes have an 
influence on the $(p+1)$th lobe. 
\begin{defn}
  Let $t_1<\dots<t_{p}$ be $p$ successive impulses times of a pulses train satisfying
  for any $i\in \llbracket 1,p \rrbracket$, $t_i\le t_{i-1} + 5$
  and $t_i-t_{i-1}\ge I_{\min}$.
  When such integer $p$ is maximal then the pulses train is said $p$-{\it persistent}.
  \label{def:proper}
\end{defn}

\begin{rem}
  Given an $1$-persistent pulses train  $(t_i)_{1\le i\le n}$ ($n>1$), 
  there exists $j\in \llbracket 1,n\rrbracket$ such that
 $t_j> t_{j-1}+5$. Then, by Property \ref{prop:ell}, $c_N^j$ is well approximated
  by  $t \mapsto (t-t_j)\, e^{-(t-t_j)}$ for $t\in [t_j,t_{j+1}]$ (the factor $R_j^\dagger$ has a negligible effect).
\end{rem}

We define now an approximation of $c_N$ denoted as $\tilde c_N$ that will be used 
to construct an approximation of the force $F$ limiting the number of terms and the error
between $c_N$ and $\tilde c_N$ is analyzed in the following proposition.
\begin{prop}
  Let  $(t_i)_{1\le i\le n}$ be a $p$-persistent pulses train. 
  Denote \[c_N(t)\coloneqq\sum_{i=0}^k R_i^\dagger \eta_i (t-t_i)\, e^{-(t-t_i)}\]
  and define its (lower) approximation by
  \[\tilde c_N(t) \coloneqq \sum_{i=\max(0,k-p+1)}^k R_i^\dagger \eta_i(t-t_i)\, e^{-(t-t_i)}\]
  for $t\in [t_k,t_{k+1}],\ k=0,\dots,n$.\\
  Then, we have: 
  \[\sup_{t\in [t_k,t_{k+1}]} \ c_N(t)-\tilde c_N(t) \le \frac{\bar R}{e}\,  
  \kappa + 5e^{-5}\bar R (k-p-\kappa+1) ,\]
  where $\lceil\cdot \rceil$ is the ceiling function and
  $\kappa=\min\left(p,\left\lceil\frac{5\tau_c}{I_{\min}}\right\rceil\right)$
  is independent of $k$.
  \label{prop:err}
  \end{prop}
\begin{pf}
  For $t\in [t_k,t_{k+1}],\, k=0,\dots,n$, we have:
  \[
    c_N(t) 
  \ge \sum_{i=k-p+1}^k R_i^\dagger \eta_i (t-t_i) e^{-(t-t_i)} =\tilde  c_N(t)
\]
and 
\[
  c_N(t) - \tilde c_N(t) 
  \le \bar R \sum_{i=0}^{k-p}  (t-t_i) e^{-(t-t_i)}.
\]
The number of indices $i\in \{0,\dots,k-2\}$ for which $t-t_i\le 5\tau_c$ is at most
 $\kappa\coloneqq\min\left(p,\left\lceil\frac{5\tau_c}{I_{\min}}\right\rceil\right)$ since
 $(t_i)_i$ is $p$-persistent and $\left\lceil\frac{5\tau_c}{I_{\min}}\right\rceil$ stands for
 the maximum number of impulse times satisfying the constraint $t_i-t_{i-1}\ge I_{\min}$ 
 in an interval of length $5\tau_c$.
\end{pf}

\begin{prop}[Tail approximation of $c_N$]
  Let $q\in \{0,\dots,n\}$.
  Denote the tail of $c_N$ by $c_N^q = c_{N\rvert_{[t_q,T]}}$
  and its average over $[t_q,T]$ by $\bar c_N^q$. We have:
  \begin{equation}
    \begin{aligned}
      \bar c_N^q= \frac{1}{T-t_q} \, \sum_{i=0}^q R_i \eta_i
      \left( \chi_i(t_q)-\chi_i(T) \right)
      +
      \frac{1}{T-t_q} \, \sum_{i=q+1}^n R_i \eta_i
      \left( 1 - \chi_i(T) \right),
    \end{aligned}
    \label{eq:cnq}
  \end{equation}
  where $\chi_i(t)=
  e^{-(t-t_i)}
  \left(1+t-t_i\right)$.
  \label{prop:cnq}
\end{prop}

\subsection{Approximations of $F$}

Integrating \eqref{eq:dotF}, the force with $F(0)=0$ can be written as 
\begin{equation}
  F(t) = A M(t)\, \int_{0}^t M^{-1}(s) m_1(s)\, \dd s, \ t\in [0,T]
  \label{eq:Fvrai}
\end{equation}
where $\displaystyle M(t) = \exp\left(-\int_{0}^t m_2(s)\, \dd s\right)$.

The following properties show that it is natural to approximate $m_1$ and 
$m_2$ by polynomial functions.
\begin{property}
Let $k\in \{0,\dots,n\}$.
  \begin{itemize}
  \item Denote $t^* = \underset{t_\in [t_k,t_{k+1}]}{\text{argmax }}\, c_N(t)$.
  Then,
     $m_1$ (resp. $m_2$)  is strictly increasing (resp. decreasing) on $[t_k,t^*]$ 
     and strictly decreasing (resp. increasing) of $[t^*,t_{k+1}]$.
     \item 
      If $t_{k+1}<t_k+2\tau_c$ then ${m_1}_{\mid [t_k,t_{k+1}]}$ is concave and 
  ${m_2}_{\mid [t_k,t_{k+1}]}$ is convex.
  \end{itemize}
\end{property}

We consider a finer partition of $(t_i)_{1\le i\le n}$ denoted 
as $(t_{i+j/p})_{0\le i\le n,\, 0\le j \le p-1}$, $p\in \mathbb{N}^*$, 
such that it satisfies 
$
t_i< t_{i+1/p} < \dots < t_{i+(p-1)/p} < t_{i+1}.
$
We approximate $m_1$ and $m_2$ on each interval $[t_{i+j/p},t_{i+(j+1)/p}]$ by a polynomial function
denoted respectively by $\tilde m_1$ and $\tilde m_2$.
\begin{exmp}[Triangular approximation of a lobe.]
  \label{ex:affine-approx}
  Since $\dot m_1 = K_m \dot c_N /(K_m+c_N)^2$ and $\dot m_2 = \tau_2\dot m_1/(\tau_1+ \tau_2m_1)^2$, 
  then $\dot m_1, \ \dot m_2$ are zero when 
  $c_N$ is maximal. 
  On $[t_{k+j/2},t_{k+(j+1)/2}]$, $j=0,1$, $m_i,\ i=1,2$ can be approximated by 
  \[\tilde m_i(t) = a_{ij,k}\,(t-t_{k+j/2})+b_{ij,k},\ k=0,\dots, n,\]
  where $t_{k+1/2} = \underset{t\in [t_k,t_{k+1}]}{\text{argmax }} \ c_N(t)$. 
  Computing, we have
  $t_{1/2} = \tau_c, \text{ and for } k=1,\dots,n:$
  \[
  t_{k+1/2} = \underset{t_\in [t_k,t_{k+1}]}{\text{argmax }}\, c_N(t) = \tau_c + 
    \frac{\sum_{i=1}^k R_i\eta_i\, t_i\, e^{t_{i}/\tau_c}}{\sum_{i=0}^k R_i\eta_i\, e^{t_{i}/\tau_c}}.
  \]
  Imposing $\tilde m_i(t_{k+j/2})=m_i(t_{k+j/2})$ and  $\tilde m_i(t_{k+(j+1)/2})=m_i(t_{k+(j+1)/2})$, we get:
  \[
    a_{ij,k} = \frac{m_i(t_{k+(j+1)/2}) - m_i(t_{k+j/2})}{t_{k+(j+1)/2}-t_{k+j/2}},
    \
    b_{ij,k} = m_i(t_{k+j/2}).
  \]
\end{exmp}

Take $k_s\in \{0,\dots,n\}$, $j_s\in \{0,\dots,p-1\}$ and $t\in [t_{k_s+j_s/p},t_{k_s+(j_s+1)/p}]$
and let $\Psi(u; i,j)$ 
be the primitive of $\tilde m_2$ 
on $[t_{i+j/p},t_{i+(j+1)/p}]$, zero at $t=t_{i+j/p}$. 
We have for  $t\in [t_{k_t+j_t/p},t_{k_t+(j_t+1)/p}]$:
\begin{equation}
  \begin{aligned}
    \tilde M(t) = 
    \exp\left(-
      \sum_{k=0}^{k_s-1} \sum_{j=0}^{p-1} 
      \left[ 
        \Psi(u;k,j)
      \right]_{t_{k+j/p}}^{t_{k+(j+1)/p}}
      -\sum_{j=0}^{j_s-1} 
      \left[ 
        \Psi(u;k_s,j)
      \right]_{t_{k_s+j/p}}^{t_{k_s+(j+1)/p}}
      -\left[ 
        \Psi(u;k_s,j_s)
      \right]_{t_{k_s+j_s/p}}^{t}
    \right),
  \end{aligned}
  \label{eq:Mkapp}
\end{equation}
and for $s\in [t_{k_s+j_s/p},t_{k_s+(j_s+1)/p}]$, $t\in [t_{k_t+j_t/p},t_{k_t+(j_t+1)/p}]$, we get: 
\begin{equation}
  \begin{aligned}
    \tilde M(t) \tilde M^{-1}(s)
    =\exp\Bigg( &-
      \left[
        \Psi(u;k_s,j_s)
      \right]_s^{t_{k_s+j_s/p}}
      +
      \sum_{j=0}^{j_s-1} \left[ 
        \Psi(u;k_s,j)
      \right]_{t_{k_s+j/p}}^{t_{k_s+(j+1)/p}}
      -
      \sum_{j=0}^{j_t-1}\left[ 
        \Psi(u;k_t,j)
      \right]_{t_{k_t+j/p}}^{t_{k_t+(j+1)/p}}
      \\
      &-
      \sum_{i=k_s}^{k_t-1}\sum_{j=0}^{p-1} 
      \left[
        \Psi(u;i,j)
      \right]_{t_{i+j/p}}^{t_{i+(j+1)/p}}
      - \left[ 
        \Psi(u;k_t,j_t)
      \right]_{t_{k_t+j_t/p}}^t
    \Bigg).
  \end{aligned}
  \label{eq:MtMsinv}
\end{equation}

To integrate the product $\tilde M(t) \tilde M^{-1}(s)\tilde m_1(s)$ with respect to $s$, 
we gather the terms depending on $s$ in \eqref{eq:MtMsinv} together and we get,
for $t\in [t_{k_t+j_t/p},t_{k_t+(j_t+1)/p}]$:
\begin{equation}
  \begin{aligned}
    &\int_{t_{k_s+j_s/p}}^{t_{k_s+(j_s+1)/p}}
    \tilde M(t)\tilde M^{-1}(s) \, \tilde m_1(s)\, \dd s
    =
    \int_{t_{k_s+j_s/p}}^{t_{k_s+(j_s+1)/p}}\exp\left( 
      \Psi(s;k_s,j_s)
    \right)\ \tilde m_1(s)\, \dd s
    \\
    &\exp \Bigg(- \Psi(t;k_t,j_t)
      +
      \sum_{j=0}^{j_s-1} 
        \Psi(t_{k_s+(j+1)/p};k_s,j)
      -
      \sum_{j=0}^{j_t-1}
        \Psi(t_{k_t+(j+1)/p};k_t,j)
      -
      \sum_{i=k_s}^{k_t-1}\sum_{j=0}^{p-1} 
        \Psi(t_{i+(j+1)/p};i,j)
    \Bigg).
  \end{aligned}
  \label{eq:intMMm1}
\end{equation}

Consequently, we obtain
an approximation of $F$ on $[t_{k_t+j_t/p},t_{k_t+(j_t+1)/p}]$, $k_t=0,\dots,n$, $j_t=0,\dots,p-1$
writing:
    %Integrate[Exp[(-1/2 a2^2 u^2 - b2 u)] (a1 u + b1),u] // Simplify //TeXForm
\begin{equation}
  \begin{aligned}
    \tilde F(t)/A0 &= 
    \int_{0}^{t}
    \tilde M(t)\tilde M^{-1}(s) \, \tilde m_1(s)\, \dd s
    \\
    &=
    \sum_{i=0}^{k_t-1}\sum_{j=0}^{p-1}
    \int_{t_{i+j/p}}^{t_{i+(j+1)/p}}
    \tilde M(t)\tilde M^{-1}(s) \, \tilde m_1(s)\, \dd s
    + 
    \sum_{j=0}^{j_t-1}
    \int_{t_{k_t+j/p}}^{t_{k_t+(j+1)/p}}
    \tilde M(t)\tilde M^{-1}(s) \, \tilde m_1(s)\, \dd s
    \\
    &+ 
    \int_{t_{k_t+j_t/p}}^{t}
    \tilde M(t)\tilde M^{-1}(s) \, \tilde m_1(s)\, \dd s.
  \end{aligned}
  \label{eq:Ftilde}
\end{equation}

\begin{prop}
Choosing $\tilde m_1$ as a piecewise polynomial function and $\tilde m_2$ as a piecewise 
constant function on $[0,T]$, the function:
\begin{equation*}
\tilde F(t) = 
    A\,\int_{0}^{t}
    \tilde M(t)\tilde M^{-1}(s) \, \tilde m_1(s)\, \dd s,
\end{equation*}
where
$\displaystyle \tilde M(t) = 
\exp\left(-\int_0^t \tilde m_2(s)\,\dd s
\right),$
has a closed-form expression in the polynomial-exponential category.
\end{prop}
\begin{pf}
Decomposing the integral \eqref{eq:Ftilde} as a sum of integrals over
the partition $(t_{i+j/p})_{ij}$ gives the 
expression:
\begin{align*}
    \tilde F(t)/A &= 
    \sum_{i=0}^{k_t-1}\sum_{j=0}^{p-1}
    \int_{t_{i+j/p}}^{t_{i+(j+1)/p}}
    \tilde M(t)\tilde M^{-1}(s) \, \tilde m_1(s)\, \dd s 
    \\
    &+ 
    \sum_{j=0}^{j_t-1}
    \int_{t_{k_t+j/p}}^{t_{k_t+(j+1)/p}}
    \tilde M(t)\tilde M^{-1}(s) \, \tilde m_1(s)\, \dd s
    \\
    &+ 
    \int_{t_{k_t+j_t/p}}^{t}
    \tilde M(t)\tilde M^{-1}(s) \, \tilde m_1(s)\, \dd s, \addtocounter{equation}{1}\tag{\theequation}
      \label{eq:Ftilde2}
\end{align*}
for $t\in [t_{k_t+j_t/p},t_{k_t+(j_t+1)/p}]$, $k_t=0,\dots,n$, $j_t=0,\dots,p-1$,
and each of this integral term belongs to the polynomial-exponential category.
\end{pf}

\begin{rem}
  To construct $\tilde F$, the functions $m_1$ and $m_2$ were considered independently in the sense that
  the approximation
  does not rely on the relation \eqref{eq:m1m2} between $m_1$ and $m_2$.
  A direct consequence is that 
  an upper approximation of the force can be obtained from an upper approximation of $m_1$.
  Outside the scope of this paper, this method may be applied for more general non-autonomous models.
\end{rem}

\begin{prop}
  Adding a real parameter $\nu$ to the functions $\tilde m_1,\tilde m_2$ as follows
  \begin{equation*}
    \overset{\asymp}{m_1}(t;\nu) = \frac{c_N(t)}{\nu\, K_m + c_N(t)} \text{ and } 	
    \overset{\asymp}{m_2}(t;\nu) = \frac{\nu}{\tau_1 + \tau_2\, m_1(t)}
  \end{equation*}
  allows to construct an upper (or lower) approximation $\tilde F$ of $F$ parameterized by $\nu$.
  \label{prop:upper}
\end{prop}

\begin{rem}
A naive approach is to use classic integration schemes to define an explicit expression for $\tilde F$.
Namely, using an explicit Euler scheme for the force equation \eqref{eq:dotF} 
adapted to the partition $(t_{i+j/p})_{ij}$
gives:
\[
\tilde F(t_{i+(j+1)/p}) = 
\tilde F(t_{i+j/p})\, c_{i,j} + A\,d_{i,j},  
\]
where $c_{ij}=(1 -  h_{i,j}\,m_2(t_{i+j/p}))$,
$h_{i,j}=t_{i+(j+1)/p}-t_{i+j/p}$
and $d_{ij}=m_1(t_{i+j/p})$ for $i=0,\dots,n$
and $j=0,\dots,p-1$.
We deduce the following explicit expression for $\tilde F(t_{k_t+j_t/p}),\ k_t=0,\dots,n$, $j_t=0,\dots p-1$:
\begin{align*}
   \tilde F(t_{k_t+j_t/p})/A &= \sum_{j=0}^{j_t-1}
    h_{k_t,j}\, d_{k_t,j} \prod_{j'=j+1}^{j_t-1}\! c_{k_t,j'}     \addtocounter{equation}{1}\tag{\theequation}
    \\
    &+\sum_{i=0}^{k_t-1}\sum_{j=0}^{p-1}
    h_{i,j}\, d_{i,j}\left(
    \prod_{j'=0}^{p-1}\prod_{i'=i+1}^{k_t-1}\! c_{i',j'}
    \prod_{j'=j+1}^{p-1}\! c_{i,j'}
    \prod_{j'=0}^{j_t-1}\! c_{k_t,j'}\right).
      \label{eq:Ftilde}
\end{align*}
However, such method is not adapted for the design of our electrostimualtor (see Section \ref{sec6}).
Indeed, it does not exploit the structure of the Hill functions $m_1$ and
$m_2$ and yields worse results -- in terms of time complexity and approximation error -- 
compared to the approximation \eqref{eq:Ftilde}. 
\end{rem}

{\bf Error estimate.}

We give a bound on the error between the approximation
$\tilde F$ and $F$ in the case where $\tilde {m_2}_{\mid [t_k,t_{k+1}]}$ is the average of $m_2$ on 
$[t_k,t_{k+1}]$.
\begin{prop}
  Consider the case where $0\le \tilde m_1(t)\le 1$ and $\tilde m_2$ is equal to the average 
  of $m_2$ on $[t_j,t_{j+1}]$, $j=0,\dots,n$.
  Assume moreover 
  that each restriction on $[t_j,t_{j+1}]$, $j=0,\dots,n$ of $m_1$ (resp. $m_2$) is concave (resp. convex).
  Then, the error between the force $F$ and its 
   approximation $\tilde F$ defined by \eqref{eq:Ftilde} satisfies for 
  $k=0,\dots,n$:\\
$\displaystyle 
    |F(t_k) - \tilde F(t_k)|/A \le 
    \int_0^{t_k} |m_1(s) - \tilde m_1(s)|\, \dd s 
+ 
    t_k \int_0^{t_k} |m_2(s)-\tilde m_2(s)|\, \dd s.
 $ \label{prop:error}
\end{prop}
\begin{proof}
  For $t_k>s$, we have:
  \begin{equation*}
    \begin{aligned}
      |F(t_k)-\tilde F(t_k)|/A_0 &=
      \left|\int_0^{t_k} M(t_k)M^{-1}(s) m_1(s) - \tilde M(t_k)\tilde M^{-1}(s) \tilde m_1(s)\, \dd s\right|
      \\
      &\le
      \int_0^{t_k} M(t_k)M^{-1}(s) |m_1(s) - \tilde m_1(s)|\, \dd s
      + \left|\int_0^{t_k} \tilde m_1(s) M(t_k)M^{-1}(s) -  \tilde M(t_k) \tilde M^{-1}(s)|\, \dd s\right|
      \\
      &\le
      \int_0^{t_k} |m_1(s) - \tilde m_1(s)|\, \dd s + 
       \left|\int_0^{t_k} M(t_k)M^{-1}(s)- \tilde M(t_k) \tilde M^{-1}(s)  \, \dd s\right|.
       \\
      &= \int_0^{t_k} |m_1(s) - \tilde m_1(s)|\, \dd s + 
      \left|\int_0^{t_k}  \exp\left( -\int_{s}^{t_k} m_2(u)\, \dd u \right) 
      - \exp\left( -\int_{s}^{t_k} \tilde m_2(u)\, \dd u \right)\,\dd s\right|
      \\
      &= \int_{0}^{t_{k}} |m_1(s) - \tilde m_1(s)|\, \dd s + 
      \left|\sum_{i=0}^{k-1} \int_{t_i}^{t_{i+1}}  \exp\left( -\int_{s}^{t_k} m_2(u)\, \dd u \right) 
      - \exp\left( -\int_{s}^{t_k} \tilde m_2(u)\, \dd u \right)\,\dd s\right|.
    \end{aligned}
    \label{eq:proof-error}
  \end{equation*}

  Recall the function $m_2$ is decreasing on $[t_i,s_i]$ and increasing on $[s_i,t_{i+1}]$
  where $s_i$ is the unique maximum of $c_N$ on $[t_i,t_{i+1}]$.
  Define $\displaystyle \xi(s)\coloneqq \tilde m_2(s)- m_2(s)$.
  Since $m_2$ is convex on $[t_i,t_{i+1}]$, we have three cases:
  
\begin{itemize}
\item[(i)] $\displaystyle \int_{t_i}^{s}\xi(u)\, \dd u\le 0$ for $s\in [t_i,t_{i+1}]$. 
We have, for $i=0,\dots,k-1$,
  \begin{equation}
    \begin{aligned} 
     \Bigg|\int_{t_i}^{t_{i+1}}  &\exp\left( -\int_{s}^{t_k} m_2(u)\, \dd u \right) 
      - \exp\left( -\int_{s}^{t_k} \tilde m_2(u)\, \dd u \right)\,\dd s\Bigg|
      \\
      &=     
      \left| \int_{t_i}^{t_{i+1}}
      \exp\left( \int_{t_{k}}^s m_2(u)\, \dd u \right) \left(1- 
        \exp\left( \int_{t_{k}}^s \xi(u)\, \dd u \right)\right)\, 
      \dd s  \right|
      \\
            &\le     
       \int_{t_i}^{t_{i+1}}
      \left|1- 
        \exp\left( \int_{t_{i}}^s \xi(u)\, \dd u \right)\, 
      \right|\dd s  , \text{ (since }  \int_{t_j}^{t_{j+1}} \xi(u)\, \dd u =0)
      \\
      &\le
      \int_{t_i}^{t_{i+1}}
       \int_{t_i}^s -\xi(u)\, \dd u\, 
     \dd s
     \\
      &\le
      (t_{i+1}-t_i)\ 
       \int_{t_i}^{t_{i+1}} |\xi(u)|\, \dd u.
    \end{aligned}
  \end{equation}
\item[(ii)] $\displaystyle \int_{t_i}^{s}\xi(u)\, \dd u\ge 0$ for $s\in [t_i,t_{i+1}]$. 
We obtain the same inequality as in the case (i) by replacing $\xi$ by $-\xi$.
\item[(iii)] There exists an unique $\theta_i \in [t_i,t_{i+1}]$ such that 
$\displaystyle \int_{t_i}^{s}\xi(u)\, \dd u\le 0$ for $s\in [t_i,\theta_{i}]$ and 
$\ge 0$ for $s\in [\theta_{i},t_{i+1}]$.
Write
  \begin{equation}
    \begin{aligned} 
     \Bigg|\int_{t_i}^{t_{i+1}}  &\exp\left( -\int_{s}^{t_k} m_2(u)\, \dd u \right) 
      - \exp\left( -\int_{s}^{t_k} \tilde m_2(u)\, \dd u \right)\,\dd s\Bigg|
      \\
      &\le   
           \Bigg|\int_{t_i}^{\theta_i}  \exp\left( -\int_{s}^{t_k} m_2(u)\, \dd u \right) 
      - \exp\left( -\int_{s}^{t_k} \tilde m_2(u)\, \dd u \right)\,\dd s\Bigg|  
      \\
      &\qquad +
                 \Bigg|\int_{\theta_i}^{t_{i+1}} \exp\left( -\int_{s}^{t_k} m_2(u)\, \dd u \right) 
      - \exp\left( -\int_{s}^{t_k} \tilde m_2(u)\, \dd u \right)\,\dd s\Bigg|  
    \end{aligned}\label{eq:case3-proof}
  \end{equation}
 We have: $\int_{\theta_i}^{s}\xi(u)\, \dd u\le 0$  for $s\in [t_i,\theta_i]$
 and  $\ge 0$ for $s\in [\theta_i,t_{i+1}]$.
 As in the case (i), the first integral in the right hand side of \eqref{eq:case3-proof} 
 is bounded by 
 $(\theta_i-t_i)\   \int_{t_i}^{\theta_i} |\xi(u)|\, \dd u$.
 Likewise, the second integral in the right hand side of \eqref{eq:case3-proof} is bounded
 by  $(t_{i+1}-\theta_i)\   \int_{\theta_i}^{t_{i+1}} |\xi(u)|\, \dd u$.
\end{itemize}  
This concludes the proof.
 \end{proof}

\section{Numerical solution to optimization problems} \label{sec5}

\subsection{Functional specification for the computation of a pulses train}

The aim is to compute a local minimum $\sigma^*=(\eta_0^*,\dots,\eta_n^*,t_1^*,\dots,t_n^*,T)\in \mathbb{R}_+^{2n+2}$ 
of a cost function denoted as $\Theta$. The free final time $T$ adjusts automatically the optimal
frequency of the pulses train. 
The functional specification of the electrostimulator imposes efficient computation of this minimum 
(real time computation) and
this prevents us (at least when $\Theta$ involves the force) from 
using direct or indirect methods such as those presented in \cite{Bakir2020},
mainly because these methods are based on 
a numerical scheme to approximate the variable $F$.

\subsection{Finite dimensional optimization methods}
\label{sec5p2}

We recall basic facts about finite dimensional optimization, see \cite{boyd2004} for details, 
to emphasize that an optimal sampled-data control
problem can be viewed as an instance of such optimization problem.

The optimization problems, associated to the optimal sampled-data control problems 
{\bf OCP1} and {\bf OCP2} presented in 
Section \ref{secOCP}, can be written in the form:
\begin{equation} \label{pb:FDP}
  \begin{aligned}
    \begin{array}{ll}
      \underset{\sigma}{\min} \quad& \Theta(\sigma)
      \\
      & \Im(\sigma) \le 0,
    \end{array}
  \end{aligned}
\end{equation}
where $\Im(\sigma) = (\Xi_1(\sigma),\dots,\Xi_{3n+5}(\sigma))$ is the vector of constraints defined by: 
\begin{equation*}
\begin{array}{ll}
  \Xi_i(\sigma^*)=t_{i-1}^*-t_{i}^*+I_{\min}, \ i=1,\dots n,
  \\
  \Xi_{n+1}(\sigma^*)=t_n^*-T, \\
  \Xi_{n+2+i}(\sigma^*)=-\eta_{i}^*, \ i=0,\dots n+1,
  \\
  \Xi_{2n+4+i}(\sigma^*)=\eta_{i}^*-1, \ i=0,\dots n+1.
\end{array}
\end{equation*}

The cost  
$
\Theta: \sigma \mapsto \Theta(\sigma)
$ related to the endurance or the force strengthening program 
is smooth with respect to $\sigma$.\\
Consider the Lagrangian defined for all $(\sigma,\mu)\in \mathbb{R}^{2n+1}\times \mathbb{R}^{3n+5}_+$ by:
\[
  \mathcal{L}(\sigma,\mu) \coloneqq \Theta(\sigma) +  \mu \cdot \Im(\sigma).
\]
The problem \eqref{pb:FDP} is equivalent to the primal problem 
\[\inf_{\sigma \in \mathbb{R}^{2n+1}} \sup_{\mu\in  \mathbb{R_+}^{3n+5}}
\mathcal{L}(\sigma,\mu)\] and the first order necessary optimality conditions 
for $\sigma^*$ to be a local minimizer,
assuming the vectors $\Xi_i'(\sigma^*),\  i\in \{i,\, \Xi_i(\sigma^*)=0\}$ to be linearly independent,
state that there exists a Lagrange multiplier $\lambda\in \mathbb{R}^{3n+5}$ such that  
\begin{align*}
&\nabla_\sigma\Theta(\sigma^*) + \lambda \cdot \Im (\sigma^*) = 0,\, \lambda \cdot \Im(\sigma^*) = 0 \\
&\lambda_i\ge 0,\ \Xi_i(\sigma^*)\le 0, \ i=1,\dots,3n+5.
\end{align*}
We usually do not solve directly these optimality conditions to compute
an optimal pair  $(\sigma^*,\lambda^*)$, but a relaxation of these conditions can lead to
efficient algorithm, namely the primal-dual interior point method \cite{boyd2004}.

\subsection{Force optimization}

We consider the problems the endurance and force strengthening optimization problems {\bf OCP1} and 
{\bf OCP2}.
For each problem, we give the approximation 
$\tilde \Theta$ of the cost functions $\Theta$ 
based on the approximation $\tilde F$ of the variable $F$ described in section \ref{sec4}.
We solve the associated problem \eqref{pb:FDP} -- where $\Theta$ is replaced by $\tilde \Theta$ --
using an interior point method 
on a standard computer\footnote{4 Intel@Core$^{TM}$ i5 CPU @ 2.4Ghz}.
Note that $\tilde \Theta$ may consist of million of bytes, for that reason it is crucial to use
an approximation of the gradient of $\tilde \Theta$ with respect to $t_{i},\ i=1,\dots,n$, computed
via finite differences (vs formal computation).
We initialize the pulses train to a regular partition of $[0,1]$ and the initial amplitudes being equal to $1$.

We consider the force approximation $\tilde F$ defined by \eqref{eq:Ftilde} taking
the piecewise affine functions $\tilde m_1, \tilde m_2$ to be equal on
$[t_{k},t_{k+1}]$, $k=0,\dots,n$ to:
\begin{equation*}
  \begin{aligned}
    &\tilde m_1(t) =
    \left\{
      \begin{array}{ll}
        m_1(t_{k+1/2}) &\text{ if } t\in [t_k,t_{k+1/2}] \\
        a_{1j,k}\, (t-t_{k+1}) +  b_{1j,k},  &\text{ if } t\in [t_{k+1/2},t_{k+1}]
      \end{array}
    \right., 
    \\
    &\tilde m_2(t) =
    \left\{
      \begin{array}{ll}
        \frac{m_2(t_{k})+m_2(t_{k+1/2})}{2}  &\text{ if } t\in [t_k,t_{k+1/2}] \\
        \frac{m_2(t_{k+1/2})+m_2(t_{k+1})}{2} ,  &\text{ if } t\in [t_{k+1/2},t_{k+1}]
      \end{array}
    \right.,
  \end{aligned}
\end{equation*}
where $t_{k+1/2} =\underset{u\in [t_k,t_{k+1}]}{\text{argmax}}  c_N(u)$,
$a_{1j,k} =(m_1(t_{k+1}) - m_1(t_{k+1/2}))/(t_{k+1}-t_{k+1/2})$ and
$b_{1j,k} = m_1(t_{k+1})$.

\subsubsection{Problem OCP1:  $\displaystyle {\Theta(\sigma) \coloneqq -F(T)}$.}

\paragraph*{Approximated cost.}

The objective function $\Theta(\sigma)=-F(T)$ is approximated by the function 
$\tilde \Theta(\sigma)=-\tilde F(T)$.
The optimization variables consist in the impulse times while the amplitudes are fixed to $1$.

{\it Numerical result: }
The optimal solution $\sigma^*$, the force response $F$ and its approximation $\tilde F$ are depicted in 
Fig.\ref{fig:maxFT20}.

\begin{figure}[htpb]
  \centering
  \includegraphics[width=0.75\columnwidth]{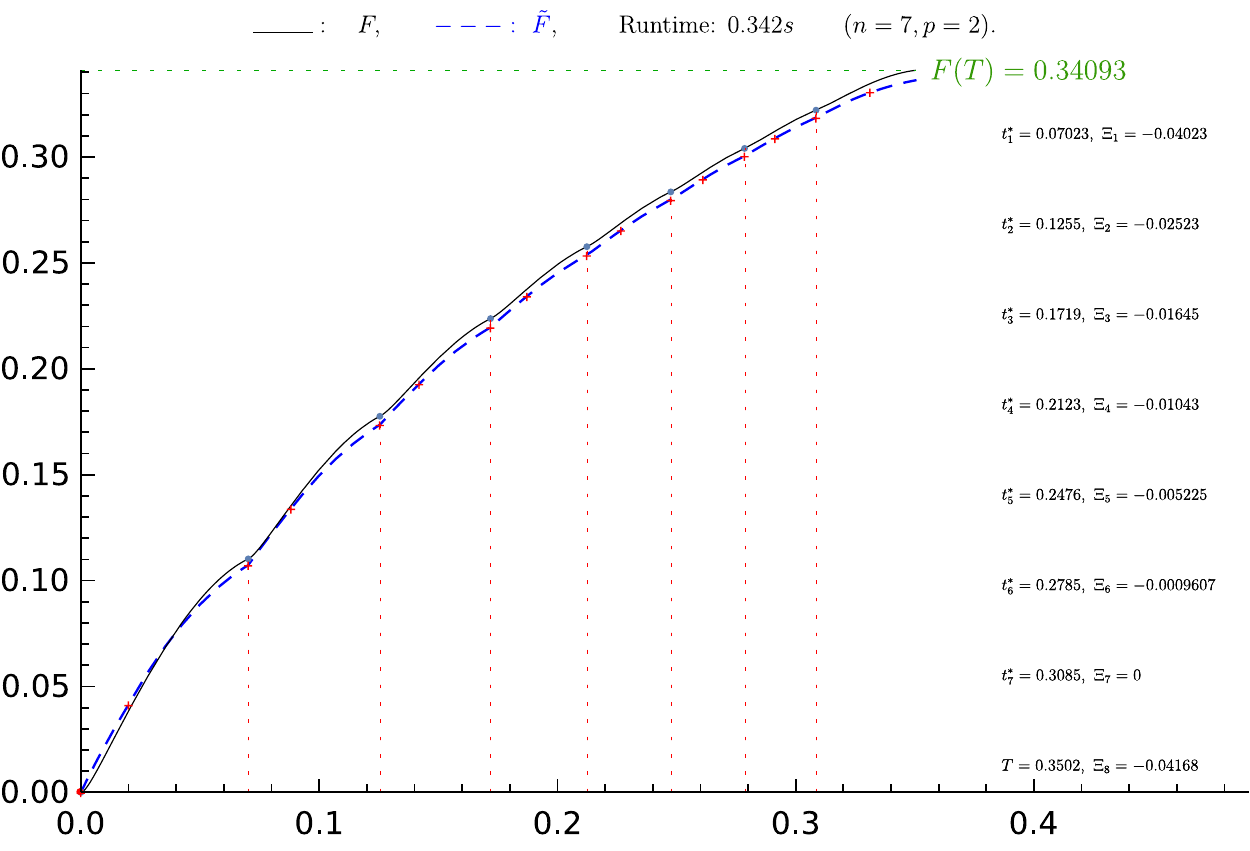}
  \caption{The dashed curve is the time evolution of $F$ 
    associated to the optimal solution $\sigma^*=(t_1^*,\dots,t_n^*,T)$ of 
    $\max_{\sigma}\, \tilde F(T)$ ($T$ free) 
    (see \eqref{eq:Ftilde} for the definition of $\tilde F$) 
    under the constraints $\Xi_i\le 0,\ i=1,\dots, n+1$ (see \eqref{pb:FDP}).   
    The continuous curve is the response $t\mapsto F(t)$ to $\sigma^*$. 
  Values of the constants are $\tau_c=20$ms, $n=7$, $I_{\min}=20$ms.  \label{fig:maxFT20}} 
\end{figure}

\subsubsection{Problem OCP2:  $\displaystyle {\Theta(\sigma) \coloneqq \int_0^T |F(t)-F_{ref}|^2\, \dd t}$.}

The cost $\displaystyle \Theta(\sigma) = \int_0^T |F(s)-F_{ref}|^2\, \dd s$ is approximated by:
\[\displaystyle \tilde \Theta(\sigma) = \sum_{k=0}^n \left( \tilde F(t_{k+1}) -F_{ref}\right)^2 (t_{k+1}-t_k),\]
where the functions $\tilde m_1$ and 
$\tilde m_2$ are replaced by $\overset{\asymp}{m_1}(t;0.95)$ and $\overset{\asymp}{m_2}(t;0.95)$ 
respectively (see Proposition \ref{prop:upper} for their definition).

{\it Numerical result: }
The optimal solution $\sigma^*$, the force response $F$ and its approximation $\tilde F$ are depicted in 
Fig.\ref{fig:minFref5_0p1}. 

\begin{figure}[htpb]
  \centering
  \includegraphics[width=0.75\columnwidth]{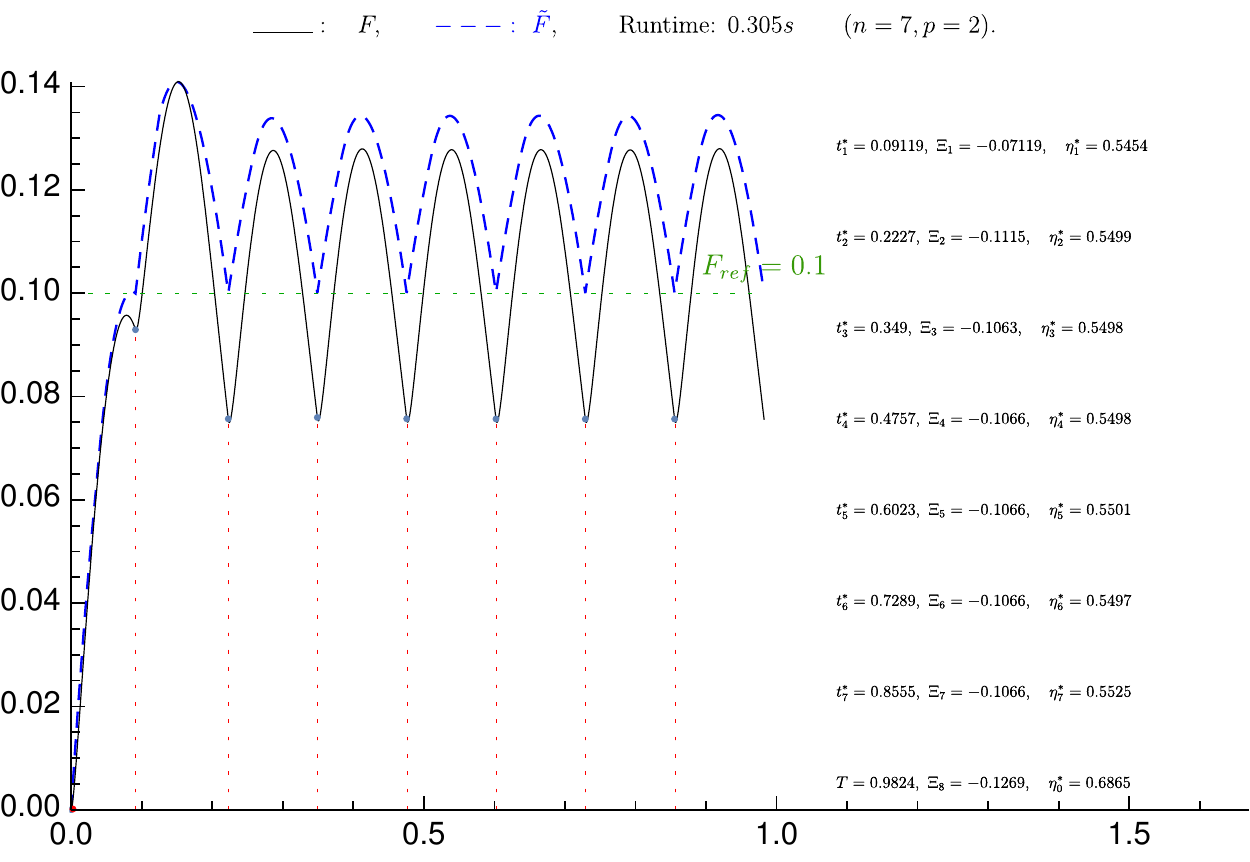}
  \caption{The dashed curve is associated to the optimal solution 
    $\sigma^*=(\eta_0^*,\dots,\eta_n^*,t_1^*,\dots,t_n^*,T)$ of 
    $\min_{\sigma}\, \sum_{k=0}^n \left( \tilde F(t_{k+1}) - F_{ref} \right)^2 (t_{k+1}-t_k)$ ($T=t_{n+1}$ is free), 
    where $\tilde F$ is the upper approximation of $F$ as described from Proposition \ref{prop:upper}
    under the constraints $\Xi_i\le 0,\ i=1,\dots, 3n+5$ (see \eqref{pb:FDP}).   
    The continuous curve is the response $t\mapsto F(t)$ to $\sigma^*$.
    Values of the constants are $\tau_c=20$ms, $n=5$, $I_{\min}=20$ms, $\nu=0.95$ and $F_{ref}=0.1$kN. 
  \label{fig:minFref5_0p1}} 
\end{figure}

\begin{figure}
  \centering
  \includegraphics[width=0.75\columnwidth]{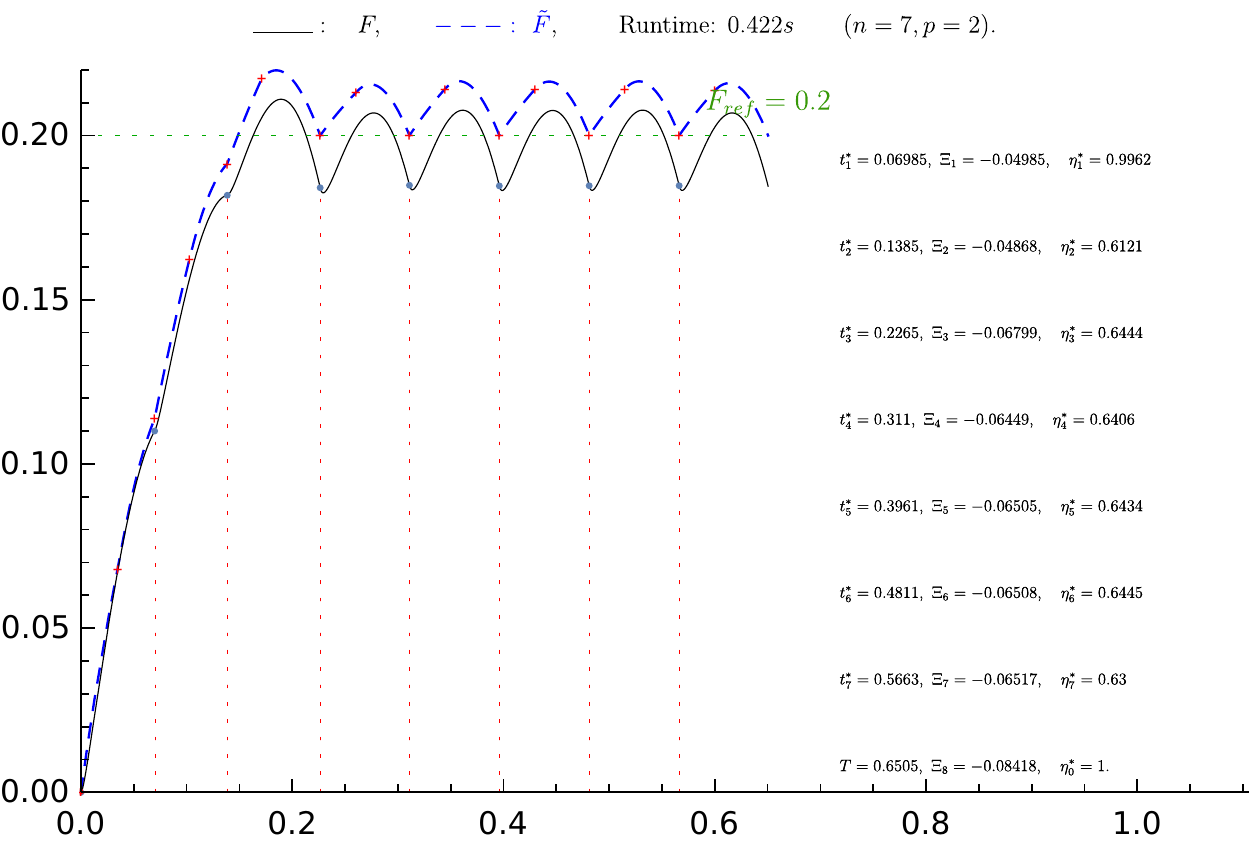}
  \caption{The dashed curve is associated to the optimal solution 
    $\sigma^*=(\eta_0^*,\dots,\eta_n^*,t_1^*,\dots,t_n^*)$ of 
    $\min_{\sigma}\, \sum_{k=0}^n \left( \tilde F(t_{k+1}) - F_{ref} \right)^2 (t_{k+1}-t_k)$ ($T=t_{n+1}$ is free), 
    where $\tilde F$ is the approximated force given by \eqref{eq:Ftilde},
    under the constraints $\Xi_i\le 0,\ i=1,\dots, 3n+5$ (see \eqref{pb:FDP}).   
    The continuous curve is the response $t\mapsto F(t)$ to $\sigma^*$.
    Values of the constants are $\tau_c=20$ms, $n=7$, $I_m=20$ms and $F_{ref}=0.2$kN. 
  \label{fig:minFref7_0p2}} 
\end{figure}

\subsection{$Ca^{2+}$ concentration optimization}

{\it Uniqueness of the optimal solution:}

Fix $\eta_i=1, i=0,\dots,n$. 
The cost function $\displaystyle \Theta(\sigma) = \int_0^T (c_N(s)-c_{ref})^2\, \dd s$ is smooth
with respect to $t_1$ and not convex on $[0,T]$. 

In Fig.\ref{fig:cout}, we plot for $n=1$ the objective function $\Theta(t_1)$ and for $T$ fixed at some 
specific values.

\begin{figure}
  \centering
  \includegraphics[width=0.75\columnwidth]{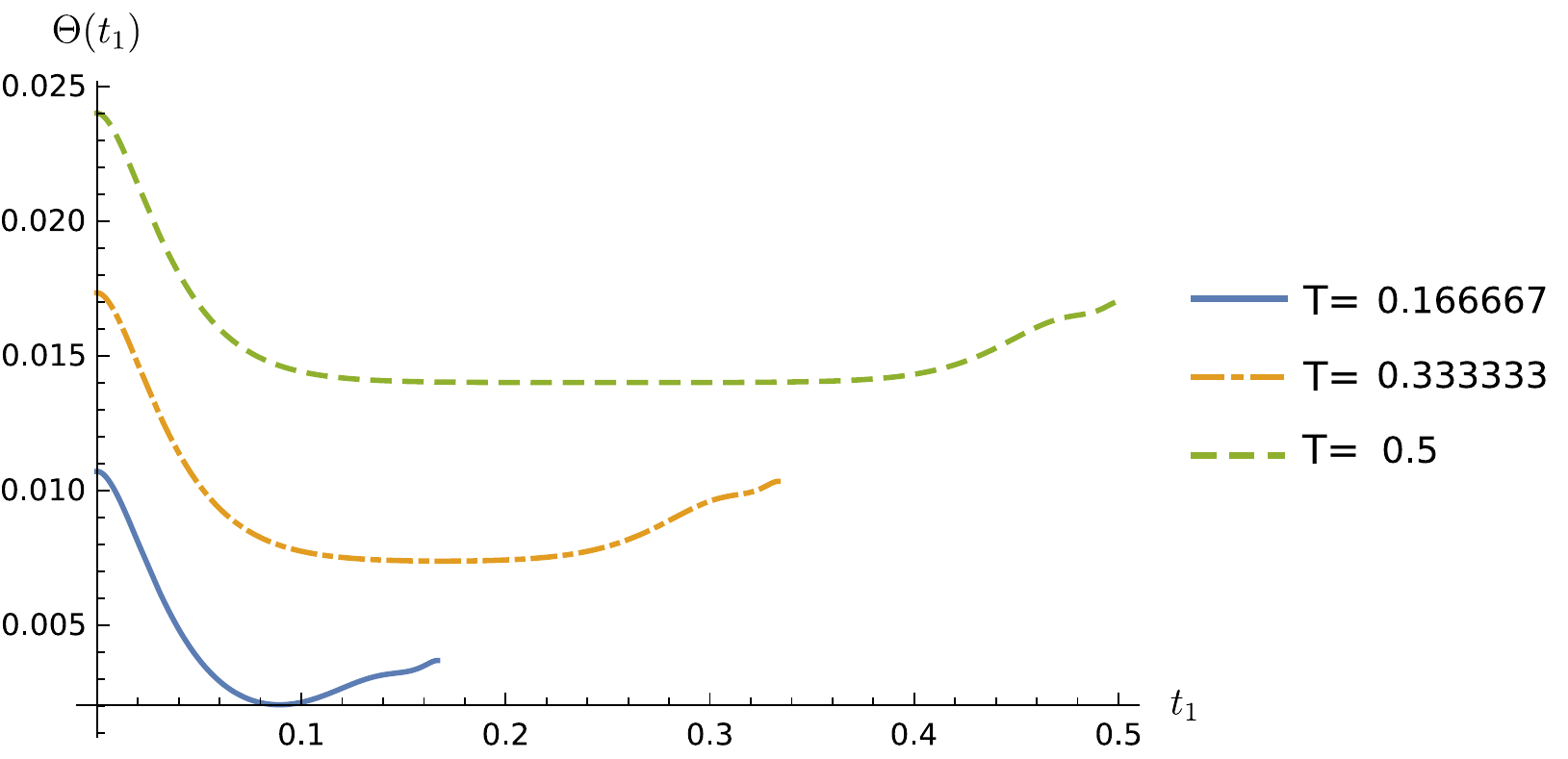}
  \caption{Objective function $\displaystyle \Theta(\sigma) = \int_0^T (c_N(s)-c_{ref})^2\, \dd s$,
  where $T$ is assigned to specific values. 
    Constants for these simulations are 
  $\tau_c=20$ms, $n=1$, $I_m=20$ms.\label{fig:cout} }
\end{figure}

The well-posedness of this optimization problem for any number $n$ of impulsions times can be shown by
inductive reasoning.

\subsubsection{Cost: ${\Theta(\sigma) = -c_N(T)}$.}

\paragraph*{True cost.}

We have an explicit expression for $c_N$, this problem can be easily solved numerically.
We consider the finite dimensional optimization problem \ref{pb:FDP} where the cost is 
\[
  \Theta(\sigma) =
  \sum_{i=0}^n R_i  (T-t_i)\,e^{-\frac{T-t_i}{\tau_c}}.
\]
(Note that the amplitudes are fixed to $1$).

{\it Numerical result:}
Fig. \ref{fig:cnT20} represents the time evolution of $c_N$ associated to (locally) optimal 
impulse times.

\begin{figure}
  \centering
  \includegraphics[width=0.75\columnwidth]{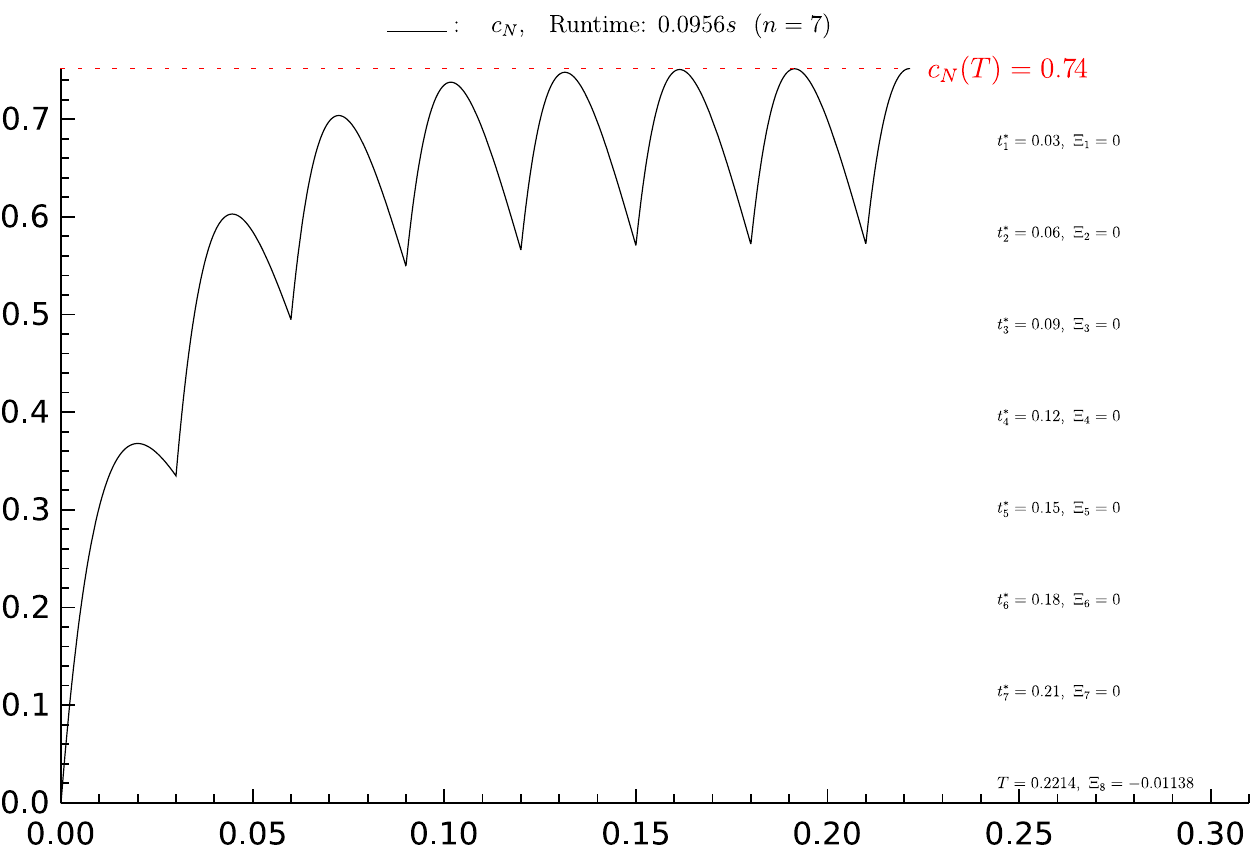}
  \caption{Time evolution of $c_N$ associated to the optimal sampling times $\sigma^*$ for 
    $\max_{\sigma}\, c_N(T)$ ($T$ free) under the constraints $\Xi_i\le 0,\ i=1,\dots, n+1$ 
    (see \eqref{pb:FDP}).
  Values of the constants are $\tau_c=20$ms, $n=7$, $I_m=20$ms.\label{fig:cnT20} }
\end{figure}

\subsubsection{Cost:  $\displaystyle {\Theta(\sigma) = \int_{0}^T |c_N(t)-c_{ref}|^2\, \dd t}$, $T$ free.}

In this case, we approximate $\Theta$ by 
\begin{equation}
  \begin{aligned}
    \tilde \Theta(\sigma) = \sum_{i=0}^n (\bar c_N^i - c_{N,ref})^2 (t_{i+1}-t_i),
  \end{aligned}
  \label{eq:approxThetaCnref}
\end{equation}
where the amplitudes and $T$ are free and $\bar c_N$ is given by Proposition \ref{prop:cnk}.

{\it Numerical result:}
The optimal solution $\sigma^*$ and its response $c_N$ are depicted in 
Fig. \ref{fig:cnref1a}.

\begin{figure}
  \centering
  \includegraphics[width=0.75\columnwidth]{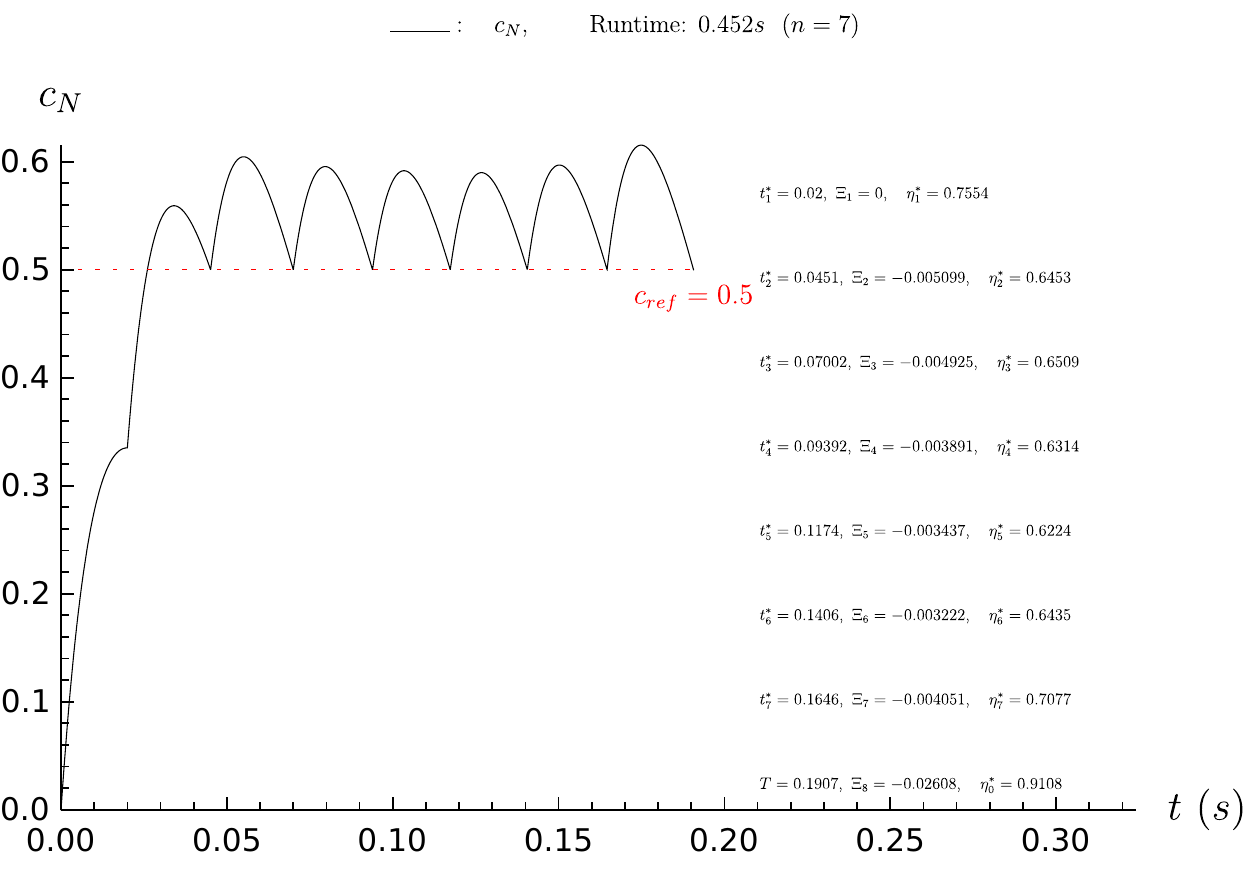}
  \caption{The dashed curve is the time evolution of $c_N$ 
    associated to the optimal solution 
    $\sigma^*=(\eta_0^*,\dots,\eta_n^*,t_1^*,\dots,t_n^*)$ of 
    $\min_{\sigma}\, \sum_{i=0}^n (\bar c_N^i - c_{N,ref})^2 (t_{i+1}-t_i)$ 
    ($T$ free), where $\bar c_N$ is defined in Proposition \ref{prop:cnk},
    under the constraints $\Xi_i\le 0,\ i=1,\dots, 3n+5$ (see \eqref{pb:FDP}).   
    The continuous curve is the response $t\mapsto c_N(t)$ to $\sigma^*$. 
    Constants for these simulations are 
  $\tau_c=20$ms, $n=7$, $I_m=20$ms and $c_{ref}=0.5$.\label{fig:cnref1a} }
\end{figure}

\section{Isometric case: design of a smart muscular electrostimulator} \label{sec6}
\label{sec6p1}

In this section we apply our study in the isometric case associated to the conception of a smart
electrostimulator. In this case, the task is to assign a constant reference force, but the general frame
 is suboptimal motion planning, see \cite{hirschorn1987,hirschorn1988} for the theoretical foundations.

Advanced commercial muscular electrostimulator for training or reeducation purposes are based on the following. 
First of all, the user defines a program training. 
Basically endurance program with low frequency sequences of trains (with constant interpulse) or force 
strengthening with high frequency trains. The program is a sequence of trains or rest periods. 
Before starting, the muscle is scanned to determine the parameters. 
A smart electrostimulator based on our study aims to design automatically such sequence, 
each program being translated into an optimization problem. 
Besides, in our framework, one can used VFT (Variable Frequency Trains) vs CFT (Constant Frequency Trains) 
in the standard case to complete the tuning of the amplitude. 

One needs the following proposition associated to the endurance program presented in Fig.\ref{fig:schema} 
to illustrate the smart electrostimulator conception.

\begin{prop} Consider the endurance program, where a reference force $F_{ref}$ is given. 
Plugging such $F_{ref}$ in $\dot{F}=0$ leads to solve the equation: 
$A\tau_2m_1^2+A\tau_1m_1-F_{ref}=0$, which has an unique positive root $m_1^+$ giving
the reference concentration $c_{N,ref}$. 
This root is stable and leads to design an optimized pulses train solving the $L^2$-optimization problem with cost: 
$\displaystyle \int_0^T |c_N(t)-c_{N,ref}|^2\, \dd t$.
\end{prop}

\begin{pf} 
Note that the mapping $m_1$: $c_N \mapsto m_1(c_N)$ is one-to-one and $m_1$ can be taken as an accessory control 
in place of $c_N$. Solving in $m_1$ the equation $\dot{F}=0$ leads to real roots denoted respectively $m_1^+>0$ 
and $m_1^-<0$. Taking $m_1^+$, stability is granted since $\lambda=-m_2(c_N^+)$ is negative, 
where $c_N^+$ is given by $m_1(c_N^+)=m_1^+$.  $\hfill \qed$
\end{pf}

\begin{rem}
The optimization problem $\min_{\sigma} \ \int_0^T |c_N(t)-c_{N,ref}|^2\, \dd t$ 
can be efficiently solved using 
the piecewise constant approximation of $c_N$ of Proposition \ref{prop:cnk}.
Indeed, we have:
\begin{equation*}
 \int_0^T |c_N(t)-c_{N,ref}|^2\, \dd t
\approx
\sum_{i=0}^n (\bar c_N^i - c_{N,ref})^2 (t_{i+1}-t_i).
\end{equation*}
\end{rem}

\section{Conclusion}
In this short article, we have mainly presented a finite dimensional approximation of the muscular force response
to FES-input exploiting the mathematical structure of the model.
The construction is based on the Ding et al. model but can be adapted to deal with the different models
discussed in \cite{wilson2011}.
We have presented one application of our study related to motion planning in the isometric case in view 
to design a smart electrostimulator, which is an ongoing industrial project.
Our approximation can be used to parameters estimation \cite{wilson2011,Stein2013} and to design
MPC-optimized sampled-data control schemes, applying standard algorithms \cite{Richalet1993,Boyd2010}
to this situation.\\
Another application of our study is to track in the non-isometric case a path in the joint angle variable 
and this will be developed in a forthcoming article.

\begin{table*}
  \renewcommand{\arraystretch}{0.9}
  \centering
  \caption{List of variables and values of the constant parameters in the Ding et al.\ model}\label{tb:params}
  \setlength{\tabcolsep}{2pt}
  \begin{tabular}{llll}
    Symbol & Unit & Value & Description \\ 
    \hline
    $C_{N}$&~---~ &~---~~&Normalized amount of $Ca^{2+}$-troponin complex\\ 
    $F$& $\text{kN}$ &~---~&Force generated by muscle\\ 
    $t_{i}$&$s$&~---~&Time of the $i^{th}$ pulse\\
    $n$&---&~---~&Total number of the pulses before time $t$\\
    $i$&---&~---~&Stimulation pulse index\\
    $\tau_{c}$&$s$&$0.02$&Time constant that commands the rise and the decay of $C_{N}$\\
    $\bar{R}$&---&$1.143$&Term of the enhancement in $C_{N}$ from successive stimuli\\
    $A$&$\text{kN}\cdot s^{-1}$&~---~&Scaling factor for the force and the shortening velocity of muscle \\
    $\tau_{1}$&$s$&$50.95$ ~&Force decline time constant when strongly bound cross-bridges absent\\
    $\tau_{2}$&$s$&$0.1244$&Force decline time constant due to friction between actin and myosin\\
    $K_{m}$&---&~---~&Sensitivity of strongly bound cross-bridges to $C_{N}$\\
    $A_{\text{rest}}$&$\text{kN}\cdot s^{-1}$&$3.009$&Value of the parameter $A$ when muscle is not fatigued\\
    $\alpha_{A}$&$s^{-2}$&$-4.0~10^{-1}$&Coefficient for the force-model parameter $A$ in the fatigue model\\
    $\tau_{fat}$&$s$&$127$&Time constant controlling the recovery of $A$\\
    \hline
  \end{tabular}
\end{table*}

\begin{figure*}
  \centering
  \includegraphics[width=1.05\linewidth]{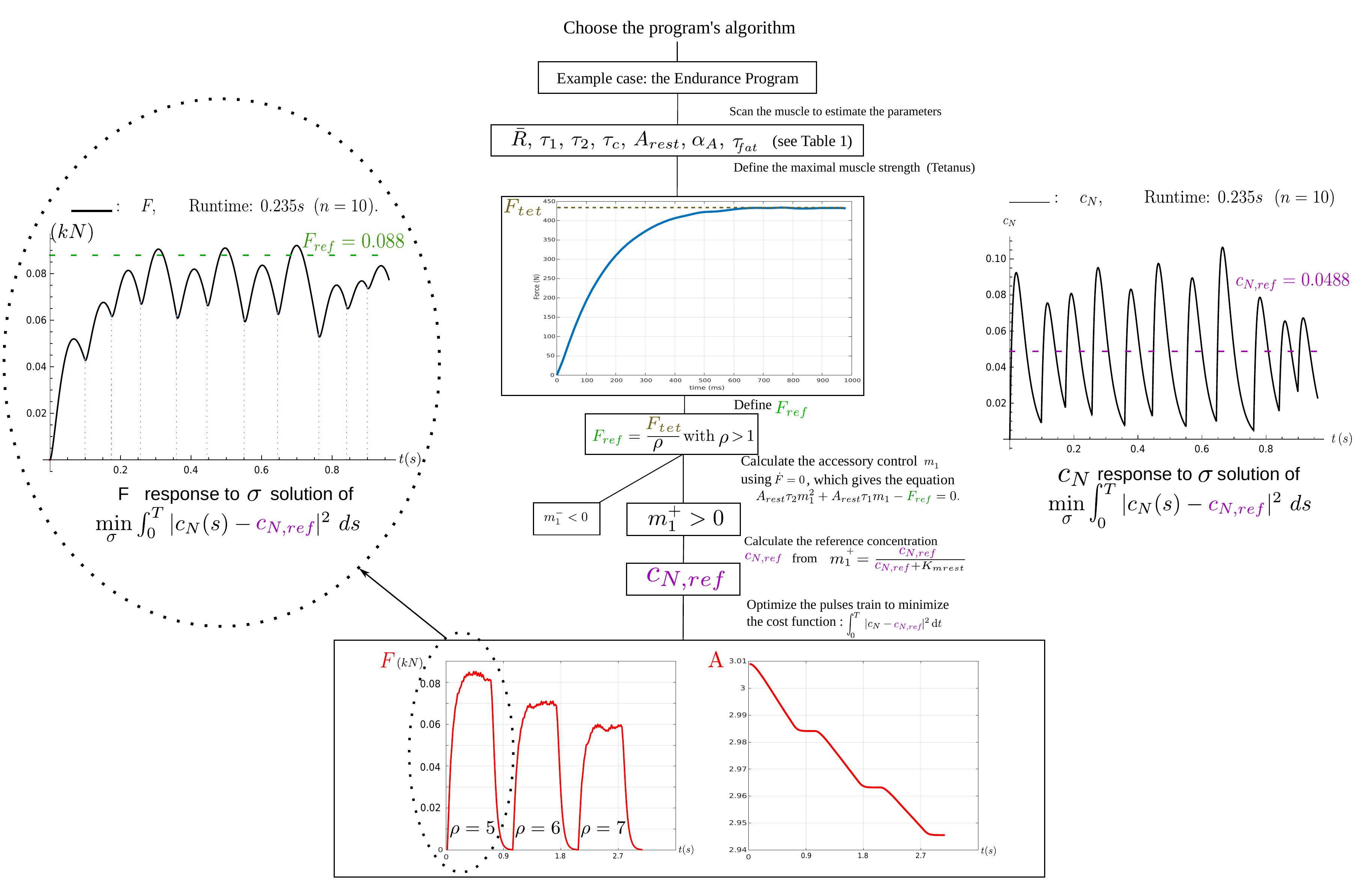}
  \caption{Endurance program for a smart electrostimulator.}
  \label{fig:schema}
\end{figure*}

\FloatBarrier


{\footnotesize
\begin{thebibliography}{99}

  \bibitem{Bakir2019} 
    \newblock Bakir T., Bonnard B. \& Rouot J. (2019).
    \newblock {A case study of optimal input-output system with sampled-data control: 
    Ding et al.\ force and fatigue muscular control model}.
    \newblock {\it Netw. Hetero. Media}, {\bf 14}(1), 79--100.

  \bibitem{Bakir2020} 
    \newblock Bakir T., Bonnard B., Bourdin L. \& Rouot J. (2020).
    \newblock {Pontryagin-type conditions for optimal muscular force response 
    to functional electrical stimulations}. 
    \newblock {\it J. Optim. Theory Appl.}, {\bf 184}, 581--602.
    
     \bibitem{bonnard2020}
    \newblock Bonnard B. \& Rouot J. (2020).
    \newblock {Geometric optimal techniques to control the muscular force response to 
    functional electrical stimulation using a non-isometric force-fatigue model}.
    \newblock  
    {\it Journal of Geometric Mechanics, American Institute of Mathematical Sciences (AIMS)}.

  \bibitem{Bourdin2016}
    \newblock Bourdin L. \& Tr\'elat E. (2016).
    \newblock {Optimal sampled-data control, and generalizations on time scales}.
    \newblock {\it Math. Cont. Related Fields}, {\bf 6}, 53--94.

  \bibitem{boyd2004} 
    \newblock Boyd S. \& Vandenberghe L. (2004).
    \newblock Convex Optimization. 
    \newblock {\it Cambridge, U.K.: Cambridge Univ. Press.}

  \bibitem{Doll2015}
    \newblock Doll B.D. , Kirsch  N.A. \& Sharma N. (2015).
    \newblock {Optimization of a Stimulation Train based on a Predictive Model of 
    Muscle Force and Fatigue}.
    \newblock {\it IFAC-PapersOnLine}, {\bf 48}(20), 338--342.

  \bibitem{Ding2000} 
    \newblock Ding J., Wexler A.S. \& Binder-Macleod S.A. (2000).
    \newblock {Development of a mathematical model that predicts optimal 
    muscle activation patterns by using brief trains}.
    \newblock {\it J. Appl. Physiol.}, {\bf 88}, 917--925.


  \bibitem{Ding2002}
    \newblock Ding J., Wexler A.S. \& Binder-Macleod S.A. (2002).
    \newblock {A predictive fatigue model. I. Predicting the effect of stimulation 
    frequency and pattern on fatigue}.
    \newblock {IEEE Transactions on Neural Systems and Rehabilitation Engineering}, 
    {\bf 10}(1), 48--58.

  \bibitem{Ding2002b}  
    \newblock J. Ding, A.S. Wexler \& S.A. Binder-Macleod,
    \newblock {A predictive fatigue model. II. Predicting the effect 
    of resting times on fatigue}.
    \newblock {IEEE Transactions on Neural Systems and Rehabilitation Engineering}, 
    {\bf 10} no.1 (2002), 59--67.

  \bibitem{Gesztelyi2012} 
    \newblock R. Gesztelyi, J. Zsuga, A. Kemeny-Beke, B. Varga, B. Juhasz \& A. Tosaki,
    \newblock  {The Hill equation and the origin of quantitative pharmacology}.
    \newblock \emph{Archive for history of exact sciences}, {\bf 66} no. 4 (2012), 427--438.

  \bibitem{hirschorn1987}
    \newblock  Hirschorn, R. M. \& Davis, J. H., 
    \newblock  {Output tracking for nonlinear systems with singular points.}
    \newblock SIAM J. Control Optim. 25 (1987), no. 3, 547--557.

  \bibitem{hirschorn1988}
    \newblock  Hirschorn, R. M. \& Davis, J. H., 
    \newblock {Global output tracking for nonlinear systems}.
    \newblock SIAM J. Control Optim. 26 (1988), no. 6, 1321--1330.


  \bibitem{Isidori1989}
    \newblock Isidori A. (1995).
    \newblock {Nonlinear Control Systems}.
    \newblock {\it 3rd ed. Berlin, Germany: Springer-Verlag}. 


  \bibitem{Marion2013} 
    \newblock Marion M.S., Wexler A.S. \& Hull M.L. (2013). 
    \newblock {Predicting non-isometric fatigue induced by electrical stimulation 
    pulse trains as a function of pulse duration}.
    \newblock {\it Journal of neuroengineering and rehabilitation}, {\bf 10}(1).

  \bibitem{menten1913} 
    \newblock Michaelis L. \&  Menten M.L. (1913). 
    \newblock {Die Kinetik der Intertinwerkung}.
    \newblock {\it Biochemische Zeitschrift}, {\bf 49} 333-369.

  \bibitem{Richalet1993}
    \newblock  Richalet J. (1993).
    \newblock  {Industrial applications of model based predictive control}.
    \newblock {\it Automatica, IFAC} {\bf 29}(5), 1251--1274.

  \bibitem{Stein2013}
    \newblock Stein R., Bucci V., Toussaint N.C., Buffie C.G., R\"atsch G., Pamer E.G. et al. (2013).
    \newblock {Ecological Modeling from Time-Series Inference: Insight into Dynamics and Stability 
    of Intestinal Microbiota}.
    \newblock {\it PLOS Computational Biology} {\bf 9}(12), 1--11.

  \bibitem{Boyd2010} 
    \newblock Wang Y. \& Boyd S. (2010).
    \newblock {Fast Model Predictive Control using Online Optimization}.
    \newblock {\it Control Systems Technology, IEEE Transactions on,} 
    {\bf 18}(2), 267--278.

  \bibitem{wilson2011}
    \newblock Wilson E. (2011).
    \newblock {Force response of locust skeletal muscle}.
    \newblock Southampton University, Ph.D. thesis.

\end{thebibliography}
}
\end{document}